\documentclass[12pt,nopagetitle]{article}
\usepackage{amsfonts, amsmath, amssymb, amsthm, enumitem}  
\usepackage[english]{babel}
\usepackage[utf8]{inputenc}
\usepackage[OT1]{fontenc}
\usepackage{bm}

\usepackage{textcomp, cmap, comment}	
\usepackage{graphicx, wrapfig}
\usepackage{epigraph, xcolor, hyperref}
\usepackage{array,tabularx,tabulary,booktabs}
\usepackage{euscript, mathrsfs}	

\newcounter{iii}

\newcommand{\bb}{{\mathcal B}}
\newcommand{\aaa}{{\mathcal A}}

\newcommand{\g}{\mathcal G}

\newcommand{\ff}{\mathcal F}
\theoremstyle{plain}
\newtheorem{thm}{Theorem}

\newtheorem{lem}{Lemma}

\theoremstyle{definition}
\newtheorem{defi}{Definition}

\date{}
\title{Non-trivial $r$-wise agreeing families}
\author{Peter Frankl\footnote{R\'enyi Institute, Budapest, Hungary; Email: {\tt peter.frankl@gmail.com}} \, and Andrey Kupavskii\footnote{Moscow Institute of Physics and Technology, St. Petersburg State University, Russia; Email: {\tt kupavskii@ya.ru} }}

\begin{document}
\maketitle
\begin{abstract}
  A family of subsets of $[n]$ is $r$-wise agreeing if for any $r$ sets from the family there is an element $x$ that is either contained in all or contained in none of the $r$ sets. The study of such families is motivated by questions in discrete optimization. In this paper, we determine the size of the largest non-trivial $r$-wise agreeing family. This can be seen as a generalization of the classical Brace-Daykin theorem.
\end{abstract}

Let $[n] = \{1,2,\ldots, n\}$ be the standard $n$-element set and $2^{[n]}$ its power set.
Let $\ff\subset 2^{[n]}$ be a family. We say that sets $F_1,\ldots, F_r\subset [n]$ {\it agree on a coordinate} $x\in[n]$ if  either $x\in\cap_{i\in[r]} F_i$ or $x\not\in\cup_{i\in[r]} F_i$. We call a family $\ff$ {\it $r$-wise $t$-agreeing} if any $r$ sets from $\ff$ agree on at least $t$ coordinates. For $t=1$ we call such families {\it $r$-wise agreeing} for shorthand. Additionally, we call $\ff$ {\it non-trivial} if $\cap_{A\in \ff}A=\emptyset$ and $\cup_{A\in \ff} A = [n]$ (that is, all sets from $\ff$ do not agree on a coordinate).

$r$-wise agreeing families appear in the context of packing and covering problems in combinatorial optimization \cite{Abdi1,Abdi2}. The connection with geometry is established through the correspondence between the subsets of $[n]$ and vertices of $\{0,1\}^n$. (This explains the use of the term `coordinate' in the paragraph above.) In particular, Abdi et al. \cite{Abdi1} proposed a conjecture that states that non-trivial $r$-wise agreeing families cannot be cube-ideal for sufficiently large $r$ (cf. \cite{Abdi1} for the definition of cube-idealness). While discussing possible strategies to attack this conjecture with Ahmad Abdi, the following question was raised:
\begin{quote}
  How big can a non-trivial $r$-wise agreeing family be?
\end{quote}
The goal of this note is to answer this question. We prove the following theorem.
\begin{thm}\label{thm1}
  Let $n>r\ge 2$ and $t\le 2^r-r-1$. Suppose that $\ff\subset 2^{[n]}$ is non-trivial $r$-wise $t$-agreeing. Then $|\ff|\le (r+t+1)2^{n-r-t}$.
\end{thm}
Let us mention that in the case $r=2, t=1$ the bound is exactly $2^{n-1}$. The proof in this case is easy: it is immediate to see that $\ff$ contains at most $1$ set out of each pair of complementary sets $A,[n]\setminus A$. The same argument shows that any $r$-wise agreeing family has size at most $2^{n-1}$. At the same time, the family of all sets not containing $1$ provides an example of an $r$-wise agreeing family of size $2^{n-1}$. If we drop the non-triviality assumption then the family with some fixed $t$ coordinates provides a lower bound of $2^{n-t}$ for the size of the largest $r$-wise $t$-agreeing family.  Theorem~\ref{thm1} shows that non-triviality forces the family to be considerably smaller.

It is natural to draw a parallel with union\footnote{A notion dual to intersecting families, which is more convenient to work with here.} families. A family $\ff\subset 2^{[n]}$ is {\it $r$-wise $t$-union} if $|A_1\cup\ldots\cup A_r|\le n-t$ for any $A_1,\ldots, A_r\in \ff$. It is {\it non-trivial} if $\cup_{A\in \ff}A=[n]$. Note that $r$-wise $t$-union families are $r$-wise agreeing, but not vice versa. $r$-wise $t$-union families are essentially $r$-wise $t$-agreeing families which must agree with coordinate value $0$. In some sense, the notion of $r$-wise agreeing families is geometrically more symmetric since it is preserved under flipping a coordinate.

\begin{thm}[Brace--Daykin \cite{BD}, Frankl \cite{Fra91,Fra19}]\label{thmbdf}
 Let $n>r\ge 2$ and $t\le 2^r-r-1$. Suppose that $\ff\subset 2^{[n]}$ is non-trivial $r$-wise $t$-union. Then
\begin{equation}\label{eqbdf}
  |\ff|\le (r+t+1)2^{n-r-t}.
\end{equation}
\end{thm}
The following example shows that the bound \eqref{eqbdf} (and therefore the bound in Theorem~\ref{thm1}) is tight:
\begin{equation}\label{eqbdfex}\bb=\{B\in 2^{[n]}: |B\cap [r+t]|\le 1\}.\end{equation}
Note that $\cup_{b\in \bb}B = [n], \cap_{b\in \bb}B = \emptyset$.  In general, it is not difficult to see that inclusion-maximal non-trivial $r$-wise $t$-union families are non-trivial $r$-wise $t$-agreeing. (Indeed, to see that it is non-trivial as an $r$-wise $t$-agreeing family, note that if  $x\in A$ for all $A\in\ff$, then we may add $A\setminus \{x\}$ to the family while keeping it $r$-wise $t$-union. This contradicts maximality.) That is, Theorem~\ref{thm1} is a strengthening of Theorem~\ref{thmbdf}.

\section{Proof of Theorem~\ref{thm1}}

The idea of the proof is to apply the squashing operation to the family. The first possibility is that the family stays nontrivial all along and then becomes down-closed. Then it is $r$-wise $t$-union, and we may apply Theorem~\ref{thmbdf}. The second possibility is that we lose non-triviality on a certain step. Then we can define a family of the same size and on the ground set of size $n-1$ that is $r$-wise $(t-1)$-agreeing (or $(r-1)$-wise agreeing for $t=1$). We then apply induction to this family. The proof strategy seems simple and natural, but it took us surprisingly long to find it.

Let us introduce a convenient definition.
\begin{defi}
  For a collection of sets $A_1,\ldots, A_r$ let $W(A_1,\ldots, A_r)$ denote the set of coordinates on which not all $A_i$'s agree: $(A_1\cup\ldots\cup A_r)\setminus (A_1\cap\ldots\cap A_r)$.  For integers $n>\ell\ge2$ and $r\ge 2$ let
$$w(n,\ell,r) = \max\{|\aaa|: \aaa\subset 2^{[n]}, |W(A_1,\ldots, A_r)|\le \ell\text{ for all }A_1,\ldots, A_r\in \aaa\}.$$
Let $w^*(n,\ell,r)$ denote the maximum taken over all non-trivial (in the agreeing sense) families $\aaa\subset 2^{[n]}$.
\end{defi}
In this terminology, Theorem~\ref{thm1} states that $w^*(n,n-t,r) = (r+t+1)2^{n-t-r}$ for $t\le 2^r-r-1$. Kleitman \cite{Kl} determined $w(n,\ell,2)$ for all $\ell$. Except for the trivial case $\ell=1$, the extremal construction is non-trivial whence $w^*(n,\ell,2) = w(n,\ell,2)$. For the proof he introduced an operation on families of sets, called {\it squashing} (cf. \cite{FT}).

For $\ff\subset 2^{[n]}$ and $i\in[n]$ define $\ff(i) = \{F\setminus \{i\}: i\in F, F\in \ff\}$ and $\ff(\bar i) =\{F: i\notin F, F\in \ff\}$. Note that $|\ff| = |\ff(i)|+|\ff(\bar i)|$ and that $\ff$ is uniquely determined by the two families $\ff(i),\ff(\bar i)\subset 2^{[n]\setminus \{i\}}$.

 For  $\ff\subset 2^{[n]}$ and $i\in[n]$ the squashed family $S_i(\ff)$ is the (unique) family $\g\subset 2^{[n]}$ determined by
$\g(i) = \ff(i)\cap \ff(\bar i)$ and $\g(\bar i) = \ff(i)\cup \ff(\bar i)$. Geometrically, one can think of this operation as of `gravity' acting along the $i$-th direction of the hypercube and making vertices corresponding to sets in $\ff$ `fall' on the hyperplane $\{x_i=0\}$ whenever the corresponding place is vacant. Note that $|\g(i)|+|\g(\bar i)| = |\ff(i)|+|\ff(\bar i)|$, implying $|\g| = |\ff|$.  The following statement is quintessential for the proofs.

\begin{lem}
  Let $\ff\subset 2^{[n]}$ and $i\in [n]$. Put $\g = S_i(\ff)$. Then for arbitrary $r$, $$\max\{|W(F_1,\ldots, F_r)|: F_1,\ldots, F_r\in \ff\}\ge \max\{|W(G_1,\ldots, G_r)|: G_1,\ldots, G_r\in \g\}.$$
\end{lem}
\begin{proof}
  Let $G_1,\ldots, G_r\in \g$. Then there exist $F_1,\ldots, F_r\in\ff$ so that $G_j\setminus \{i\} = F_j\setminus \{i\}$ for $j\in[r]$. Hence $W(F_1,\ldots, F_r)\subset W(G_1,\ldots, G_r)$ is automatically satisfied unless $i\notin W(F_1,\ldots, F_r)$ and $i\in W(G_1,\ldots, G_r)$. The latter implies $i\in G_1\cup\ldots \cup G_r$ and $i\notin G_1\cap \ldots\cap G_r$. By symmetry, we may assume $i\in G_1, i\notin G_2$.

Since $\g(i) = \ff(i)\cap \ff(\bar i)$, both $G_1$ and $G_1\setminus \{i\}$ must be members of $\ff$. Hence no matter whether $i\in F_2$ or $i\notin F_2$, we may choose $F_1\in \{G_1\setminus \{i\}, G_i\}\subset \ff$ so that $i\in (F_1\cup F_2)\setminus (F_1\cap F_2)$ whence $i\in W(F_1,\ldots, F_r)$. Consequently, $W(F_1,\ldots, F_r)\supset W(G_1\ldots, G_r)$. This proves the lemma.
\end{proof}

The problem with squashing is that it might destroy non-triviality. E.g., for $\ff_{even} = \{F\subset [n]: |F|\text{ is even }\}$, which is an extremal constriction for $w^*(n,n-1,2)$ if $n$ is odd, $S_i(\ff) = 2^{[n]\setminus \{i\}}$ for all $i\in[n]$. The following lemma permits us to circumvent this difficulty.

\begin{lem}\label{lem2}
  Let $\ff\subset 2^{[n]}$ be a non-trivial $r$-wise $t$-agreeing family. For $j\in[n]$, consider $\ff_j:=\{F\setminus \{j\}: F\in \ff\}$, thought of as a subfamily of $2^{[n]\setminus \{j\}}$. If $t\ge 2$ then $\ff_j$ is non-trivial $r$-wise $(t-1)$-agreeing. If $t=1$ then $\ff_j$ is non-trivial $(r-1)$-wise agreeing.
\end{lem}
\begin{proof}
It should be clear that $\cap_{A\in \ff_j}A\subset \cap_{A\in \ff} A = \emptyset$ and $\cup_{A\in \ff_j}A = \cup_{A\in \ff}A \setminus \{j\} = [n]\setminus \{j\}$, and thus $\ff_j$ is non-trivial. By the definition of $\ff_j$, for any $\ell$ and  $A_1,\ldots, A_\ell\in \ff_j$ there are $B_1,\ldots, B_\ell\in \ff$ such that $B_i\in \{A_i,A_i\cup \{j\}\}$, and thus
$$W(B_1,\ldots  B_\ell)\setminus \{j\}= W(A_1,\ldots  A_\ell).$$
If $t\ge 2$, then, using the above for $\ell=r$, we get $(n-1)-|W(A_1,\ldots  A_r)|\ge n-|W(B_1,\ldots  B_r)|-1\ge t-1.$  If $t=1$ then, using the above for $\ell = r-1$, we get $W(B_1,\ldots  B_{r-1})\setminus \{j\}= W(A_1,\ldots  A_{r-1}).$ Since $\ff$ is non-trivial, we can find a set $B_r$ such that $B_1,\ldots, B_r$ do not agree on $j$, and thus, using that $\ff$ is $r$-wise agreeing, we get $\emptyset \ne [n]\setminus W(B_1,\ldots, B_r)\supset [n]\setminus \big(W(B_1,\ldots  B_{r-1})\cup \{j\}\big)= ([n]\setminus \{j\})\setminus W(A_1,\ldots  A_{r-1})$, which proves that $\ff_j$ is $(r-1)$-wise agreeing.
\end{proof}

\begin{proof}[Proof of Theorem~\ref{thm1}]
The proof is by induction on $t+r$, subject to the constraint $t\le 2^{r}-r-1$. The case $r=2$ serves as the base case (note that here only $t=1$ is allowed).

Take the largest non-trivial $r$-wise $t$-agreeing family $\ff\subset 2^{[n]}$ and sequentially apply the squashing operations to $\ff$ for $j=1,2,\ldots, n$. There are two possible outcomes of this procedure. The first outcome is that the family (by abuse of notation, also $\ff$) always stays non-trivial. The resulting family is down-closed: for any set $A\in \ff$ and $B\subset A$, we have $B\in \ff$. (Informally, squashing in the $i$-th coordinate makes family down-closed with respect to coordinate $i$ while not affecting this property w.r.t. other coordinates.)  
Then, whenever $x\in A_1\cap\ldots\cap A_r$, $A_i\in \ff$, we also have $A_1\setminus \{x\}\in \ff$, and the set of agreeing coordinates for $A_1\setminus \{x\}, A_2,\ldots, A_r$ does not include $x$. This implies that, in order to guarantee the $r$-wise $t$-agreeing property, we must have $|A_1\cup \ldots\cup A_r|\le n-t.$ In other words, $\ff$ is a non-trivial $r$-wise $t$-union family, and we may apply Theorem~\ref{thmbdf} to $\ff$ and get the desired bound $|\ff|\le (r+t+1)2^{n-r-t}$.

The second outcome is that at a certain stage we lose non-triviality: while $\ff$ is non-trivial,  $S_j(\ff)$ is trivial. This means that no set in $S_j(\ff)$ contains $j$,  and thus $S_j(\ff)$ coincides with $\ff_j$ (as defined in Lemma~\ref{lem2}), in particular, $|S_j(\ff)| = |\ff_j|$. By Lemma~\ref{lem2}, $\ff_j$ is  non-trivial $r$-wise $(t-1$)-agreeing for $t\ge 2$, and non-trivial $(r-1)$-wise agreeing for $t=1$. In any case, we may apply the induction hypothesis to $\ff_j$ and get $$|\ff|=|S_j(\ff)|=|\ff_j|\le (r+t)2^{(n-1)-r-t+1} <(r+t+1)2^{n-r-t},$$ which proves the required bound.
\end{proof}

Working a bit harder, one can determine the families for which equality in Theorem~\ref{thm1} for $r\ge 3$ is attained.

\begin{thm} Suppose that $r\ge 3,t\ge 1$ and $\ff$ is $r$-wise $t$-agreeing for $t<2^{r}-r-1$ and $|\ff| = (r+t+1)2^{n-r-t}$. Then there is a set $A\in {[n]\choose r+t}$ and $R\subset A$ such that
\begin{equation}\label{eqex2}\ff = \{F\Delta R: F\subset [n], |F\cap A|\le 1\},\end{equation}
where $\Delta$ stands for symmetric difference.
\end{thm}
To prove this theorem, we  analyze the squashing procedure. If, while running the procedure, we lose non-triviality, then the size of the family is smaller than the extremal value, and thus we must arrive at a non-trivial $r$-wise $t$-union family at the end of the procedure. The first author showed \cite{Fra19} that the extremal family in Theorem~\ref{thmbdf} for $t<2^r-r-1$ is unique and, up to permuting the coordinates, is of the form \eqref{eqbdfex}. Thus, $\ff$ is also of the form \eqref{eqex2} at the end of the procedure. In order to complete the proof, we need to show that, provided $S_j(\ff)$ is of the form \eqref{eqex2}, $\ff$ itself must be of the form \eqref{eqex2}. In order to simplify the exposition, let us assume that $j = 1$ and $A = [r+t]$. Next, replacing $\ff$ with $\ff\Delta R =\{F\Delta R: F\in\ff\}$ preserves the property of being non-trivial $r$-wise $t$-intersecting and transforms a family of the form \eqref{eqex2} into a family of the same form. Thus, we may replace $\ff$ with $\ff\Delta R'$  a suitably chosen $R'\subset [2,r+t]$ and note that $S_1(\ff\Delta R') = S_1(\ff)\Delta R'$.  We may thus w.l.o.g. assume that $S_1(\ff) = \{ F\subset [n]: |F\cap[r+t]|\le 1\}$.

By the definition of squashing, we must have $\ff(1)\cap \ff(\bar 1)= \{F\subset [2,n]: F\cap[2,r+t]=\emptyset\}$ and $\ff(1)\cup \ff(\bar 1) = \{F\subset [2,n]: |F\cap[2,r+t]|\le 1\}$. Note that  $$\ff(1)\Delta \ff(\bar 1)=\{F\subset [2,n]: |F\cap [2,r+t]|=1\}.$$
 If either $\ff(1)\supset \ff(\bar 1)$ or $\ff(\bar 1)\supset \ff(1)$ then $\ff$ is of the form \eqref{eqex2}. Arguing indirectly, assume that both $\ff(1)\setminus \ff(\bar 1)$ and $\ff(\bar 1)\setminus \ff(1)$ are non-empty.  Further, assume w.l.o.g. that $|\ff(1)\setminus \ff(\bar 1)|\le |\ff(\bar 1)\setminus \ff(1)|$ and take $A_1\in \ff(1)\setminus \ff(\bar 1)$. Assume that $A_1\cap [2,r+t]=\{i_1\}$. Take $A_2\in \ff(\bar 1)\setminus \ff(1)$ such that $A_2\cap [2,r+t]=\{i_2\}$, $i_2\ne i_1$. This is possible since
{\small $$|\ff(\bar 1)\setminus \ff(1)|\ge \frac 12|\ff(1)\Delta \ff(\bar 1)| = \frac 12 (r+t-1)2^{n-r-t}>2^{n-r-t}=\big|\{F\subset [2,n]: F\cap[2,r+t] = \{i_1\}\}\big|,$$}
\noindent where in the second inequality we use the assumption $r\ge 3$. The key observation is that, if we denote $A_1', A_2'$ the sets that correspond to $A_1,A_2$ in $\ff$, then $A_1'\cap A_2'\cap [r+t] = \emptyset$ and $|(A_1'\cup A_2')\cap[r+t]| =3$. Crucially, the set of coordinates inside $[r+t]$ where they do not agree is $3$ and not $2$, as it would have been for any two sets in the extremal example. In the next paragraph, we complete these two sets to an $r$-tuple that violates the $r$-wise $t$-agreeing property.

Next, fix some distinct $i_3,\ldots, i_{r}\in [2,r+t]\setminus\{i_1,i_2\}$ and for each $i_s,$ $s\in[3,r]$, take a set $A_s\in \ff(\bar 1)\cup \ff(1)$ such that $A_s\cap [2,r+t] = \{i_s\}$ and $A_s\cap [r+t+1,n]=[r+t+1,n]\setminus A_1$. For each $s\in [r]$ let $A'_i$ be the  set in $\ff$ that corresponds to $A_i$. Note that $\cup_{i\in[r]}A'_i\cap [2,n] = \cup_{i\in[r]}A_i\cap [2,n] = \{i_1,\ldots, i_r\}\cup[r+t+1,n]$ and that $1\in A'_1\cup A'_2\subset \cup_{i\in[r]}A'_i$. Thus, $|\cup_{i\in[r]}A'_i|=n-t+1$. Similarly,  $\cap_{i\in[r]}A'_i\cap [2,n] = \cap_{i\in[r]}A_i\cap [2,n] = \emptyset$ and $1\not\in A'_1\cap A'_2\supset \cap_{i\in[r]}A'_i$. Thus, $\cap_{i\in[r]}A'_i=\emptyset$. We conclude that $|W(A'_1,\ldots, A'_r)|=n-t+1$, contradicting the fact that $\ff$ is $r$-wise $t$-agreeing.\\

{\bf Acknowledgements} We thank the referees for their detailed comments that helped to improve the presentation of the paper. The research is supported by the Ministry of Science and Higher Education of the Russian Federation (Goszadaniye) No. 075-03-2024-117, project No. FSMG-2024-0025.

\end{document}